\documentclass{amsart}
\usepackage{amssymb}
\usepackage{graphicx}
\def\({\left(}
\def\){\right)}

\newtheorem{lema}{Lemma}[section]

\newtheorem*{teorema*}{Theorem}

\newtheorem{remark}[lema]{Remark}

\newtheorem{lemma}{Lemma}[section]

\newtheorem{theorem}[lema]{Theorem}

\newtheorem{proposition}[lema]{Proposition}

\newtheorem{definition}[lema]{Definition}

  {\par \hfill \fbox{}}
  {\par \hfill \fbox{}}
  {\par \hfill }
  {\par \hfill }

\def\beq{\begin{equation}}
\def\eeq{\end{equation}}

\def\epsilon{\varepsilon}

\begin{document}

\title{Some new results in extended cone $b-$metric space}

\author{Iqbal Kour}
\address{Department of Mathematics, University of Jammu } 
\email{iqbalkour208@gmail.com}

\author{Pooja Saproo}
\address{Department of Mathematics, University of Jammu} 
\email{poojasaproo1994@gmail.com}

\author{Shallu Sharma*}
\address{Department of Mathematics, University of Jammu } 
\email{shallujamwal09@gmail.com}

\keywords{Cone metric space, extended cone b-metric space, fixed point, completeness }

\date{}
%\dedicatory{}
%\commby{}

\begin{abstract}
The main aspect of this paper is to investigate some topological properties and Kannan-type contractions in extended cone $b-$metric spaces. Additionally we have imposed some extra conditions such that a sequence in an extended cone $b-$metric space becomes a Cauchy sequence. Furthermore, in order to achieve new results the concept of asymptotic regularity has also been utilized.
\end{abstract}
\subjclass{47H10, 49J45, 54H25, 54E35,54A20}

% ----------------------------------------------------------------------
%\begin{center}
\maketitle{ }
%\end{center}
% ----------------------------------------------------------------------
\section{Introduction}
The contraction mapping principle in many forms of generalised metric spaces serves as the foundation for advanced metric fixed point theory. b-metric space is one of such generalised metric spaces, which was first introduced by Bakhtin \cite{B} and Czerwik \cite{C}. The renowned Banach  contraction principle in b-metric space was generalised by Bakhtin \cite{B}.  Extended b-metric was introduced by Kamran et al. \cite{KSU} as a development of the b-metric. Aydi et al. \cite{AFKKU} introduced the notion of a new extended $b-$metric space. Guang et al. \cite{HZ} created cone metric space in 2004, replacing the set of real numbers with an ordered Banach space. Kannan \cite{K} presented one of the most significant generalisations of the Banach contraction principle. Kannan's work \cite{K} enhanced the Banach contraction mapping notion by presenting a new contraction, currently known as the Kannan contraction. Kannan fixed point results have been extended and generalised in the establishment of $b$-metric spaces \cite{C} and generalised metric spaces \cite{AA}.
Hussain et al.\cite{HS} established the concept of cone $b$-metric space by combining the notions of $b-$metric and cone metric. Das and Beg \cite{DB} further introduced the notion of extended cone $b$-metric space. In this paper we investigate Kannan type contractions within the framework of new extended b-metric.    
\section{Preliminaries}
\begin{definition}\cite{C}
Let $\mathsf X\neq \phi$ be any set and  $t\in [1,\infty).$ A $b-$metric is a function $\mathfrak B:\mathsf X\times \mathsf X\to [0,\infty)$ such that for every $x,y,z\in \mathsf X$ the following hold:
\begin{itemize}
\item [(i)] $\mathfrak B(x,y)=0 ~\iff ~x=y,$
\item [(ii)] $\mathfrak B(x,y)=\mathfrak B(y,x),$
\item [(iii)] $\mathfrak B(x,z)\leq t[\mathfrak B(x,y)+\mathfrak B(y,z)].$
\end{itemize} 
Then $(\mathsf X,\mathfrak B)$ is said to be a $b-$metric space.
\end{definition}

\begin{definition}\cite{AFKKU}
Let $\mathsf X\neq \phi$ be any set and $\Theta:\mathsf X \times \mathsf X \times \mathsf X \to [1,\infty)$ be a function. A map $d:\mathsf X\times \mathsf X\to [0,\infty)$ is said to be an extended $b-$metric on $\mathsf X$ if for every $x,y,z\in \mathsf X$ the following hold:
\begin{itemize}
\item [(i)] $d(x,y)=0 ~\iff ~x=y,$
\item [(ii)] $d(x,y)=d(y,x),$
\item [(iii)] $d(x,z)\leq \Theta(x,y,z)(d(x,y)+d(y,z)).$
\end{itemize} 
Then $(\mathsf X,\mathsf d)$ is said an extended $b-$metric space. 
\end{definition}

\begin{definition}\cite{HZ}
Let $\mathsf E\subset \mathbb R$ be a Banach space. A set $\mathsf P$ contained in $\mathsf E$  is said to be a cone if it satisfies the following:
\begin{itemize}
\item [(i)] $\mathsf P$ is non-empty, closed and $\mathsf P\neq \{0\}.$
\item [(ii)]$sx+ty\in \mathsf P$ whenever $x,y\in \mathsf P$ and $s,t\in \mathbb{R}_{\geq 0}.$
\item [(iii)] $x\in \mathsf P$ and $-x\mathsf \in P$ implies that $x=0.$
\end{itemize}
\end{definition}
Hereafter we assume that $\mathsf P$ is a cone contained in a real Banach space $\mathsf E.$
\begin{definition}\cite{HZ}
A partial ordering $\leq$ with respect to $\mathsf P$ is defined as:
\begin{itemize}
\item [(i)] $x\leq y \iff y-x\in \mathsf P$ and $x<y$ denotes that $x\leq y,x\neq y.$
\item [(ii)] $x\ll y$ indicates that $y-x$ is an element of int$(\mathsf P)$(interior $\mathsf P$). 
\end{itemize}
\end{definition}	

\begin{definition}\cite{HZ}
$\mathsf P$ is said to be normal if there is a positive number $\mathsf N$ such that $x\leq y$ implies that $\|x\|\leq \mathsf N \|y\|,\forall x,y\in \mathsf E.$ The smallest $\mathsf N$ satisfying the above condition of normality is said to be the normal constant of $\mathsf P.$ 
\end{definition}

Hereafter we assume that $\mathsf P$ has non-empty interior and $\leq$ denotes the partial ordering with respect to $\mathsf P.$

\begin{definition}\cite{HZ}
Let $\mathsf X\neq \phi$ be any set. A cone metric on $\mathsf X$ is a map $d:\mathsf X\times \mathsf X \to \mathsf E$ such that for each $x,y,z\in \mathsf X$ the following conditions are satisfied:
\begin{itemize}
\item [(i)] $d(x,y)> 0,d(x,y)=0 ~\iff ~x=y,$
\item [(ii)] $d(x,y)=d(y,x),$
\item [(iii)] $d(x,z)\leq d(x,y)+d(y,z).$
\end{itemize} 
Then $(\mathsf X,\mathsf d)$ is a cone metric space.
\end{definition}

\begin{definition}\cite{HZ}
Let $\{x_{n}\}$ be a sequence in a cone metric space $(\mathsf X,d)$ and $\mathsf P$ be a normal cone with normal constant $N.$ Then 
\begin{itemize}
\item [(i)] $\{x_{n}\}$ is convergent to $x$ if for each $u\in \mathsf E,0\ll u,$ there is a natural number $\mathsf M$ so that for all $m\geq \mathsf M,$ we have $d(x_{n},x)\ll u.$ 
\item [(ii)]  $\{x_{n}\}$ is a Cauchy sequence if for each $u\in \mathsf E,0\ll u,$ there is a natural number $\mathsf M$ so that for all $m,n\geq \mathsf M,$ we have $d(x_{n},x_{m})\ll u.$  
\item [(iii)] $(\mathsf X,d)$ is complete if every Cauchy sequence converges in $\mathsf X.$
\end{itemize}
\end{definition}

\begin{lemma}\cite{RDH}\label{lem:RDH}
Let $(\mathsf X,d)$ be a cone metric space.
\begin{itemize}
\item [(i)] For every $u_{1}\gg 0$ and $u_{2}\in \mathsf P$ there exists $u_{3}\gg 0$ such that $u\gg u_{1}$ and $u\gg u_{2}.$
\item [(ii)] For every $u_{1}\gg 0,u_{2}\gg 0$ there exists $u\gg 0$ such that $u\ll u_{1}$ and $u\ll u_{2}.$
\end{itemize}
\end{lemma}

\begin{definition}\cite{DB}
Let $\mathsf X\neq \phi$ be any set and $\Theta:\mathsf X \times \mathsf X \times \mathsf X \to [1,\infty)$ be a function. A map $d_{\Theta}:\mathsf X\times \mathsf X\to \mathsf E$ is said to be an extended cone $b-$metric on $\mathsf X$ if for every $x,y,z\in \mathsf X$ the following hold:
\begin{itemize}
\item [(i)] $d_{\Theta}(x,y)>0,d_{\Theta}(x,y)=0 ~\iff ~x=y,$
\item [(ii)] $d_{\Theta}(x,y)=d_{\Theta}(y,x),$
\item [(iii)] $d_{\Theta}(x,z)\leq \Theta(x,y,z)(d_{\Theta}(x,y)+d_{\Theta}(y,z)).$
\end{itemize} 
Then $(\mathsf X,\mathsf d_{\Theta})$ is said an extended cone $b-$metric space. 
\end{definition}

\begin{definition}\cite{DB}
Let $\{x_{n}\}$ be a sequence in an extended cone $b-$metric space $(\mathsf X,d_{\Theta})$ and $\mathsf P$ be a normal cone with normal constant $N.$ Then 
\begin{itemize}
\item [(i)] $\{x_{n}\}$ is convergent to $x$ if for each $u\in \mathsf E,0\ll u,$ there is a natural number $\mathsf M$ so that for all $m\geq \mathsf M,$ we have $d_{\Theta}(x_{n},x)\ll u.$ 
\item [(ii)] $\{x_{n}\}$ is a Cauchy sequence if for each $u\in \mathsf E,0\ll u,$ there is a natural number $\mathsf M$ so that for all $m,n\geq \mathsf M,$ we have $d_{\Theta}(x_{n},x_{m})\ll u.$  
\item [(iii)] $(\mathsf X,d_{\Theta})$ is complete if every Cauchy sequence in $\mathsf X$ converges in $\mathsf X.$
\end{itemize}
\end{definition}

\begin{definition}\cite{DB}
Let $(\mathsf{X},d_{\Theta})$ be an extended cone $b-$metric space and $x\in \mathsf X,0\ll u.$ The open and closed balls in $\mathsf X$ are defined as $\mathsf B(a,u)=\{y\in \mathsf X:d_{\Theta}(a,y)\ll u\}$ and $\mathsf B[a,y]=\{y\in \mathsf X:d_{\Theta}(a,y)\leq u\},$ respectively.
\end{definition}

\begin{definition}\cite{DB}
Let $(\mathsf X,d_{\Theta})$ be an extended cone $b-$metric and $\{(x_{n},y_{n})\}$ be a sequence in $\mathsf X \times \mathsf X$. Then $d_{\Theta}$ is continuous if $x_{n}$ is convergent to $x$ and $y_{n}$ is convergent to $y$ implies that  $d_{\Theta}(x_{n},y_{n})$ is convergent to $d_{\Theta}(x,y)$ in $\mathsf E.$ 
\end{definition}

\section{Topological properties of extended cone $b-$metric space}
\begin{proposition}
The family $\mathcal B=\{\mathsf B(x,u):u\gg 0\}$ is a basis for the topology $\sigma_{d_{\Theta}}$ on $\mathsf X.$ 
\end{proposition}

\begin{proof}
\begin{itemize}
\item [(i)] Let $x\in \mathsf X.$ Then there exists $u\gg 0$ such that $x\in \mathsf B(x,u).$ Hence $x\in \mathsf B(x,u)\subseteq \displaystyle \bigcup_{x\in \mathsf X,u\gg 0} \mathsf B(x,u),$
\item [(ii)] Let $x\in \mathsf X,u_{1}\gg 0,u_{2}\gg 0$ such that $x\in \mathsf B(x,u_{1})\cap \mathsf B(x,u_{2}).$ Then by using Lemma \ref{lem:RDH}  there exists $u\gg 0$ such that $u\ll u_{1}$ and $u\ll u_{2}.$ Clearly, $x\in \mathsf B(x,u)\subseteq \mathsf B(x,u_{1})\cap \mathsf B(x,u_{2}).$   
\end{itemize}
\end{proof}

\begin{definition}
Let $(\mathsf{X},d_{\Theta})$ be an extended cone $b-$metric space. A set $\mathfrak U\subset (\mathsf X,d_{\Theta})$ is said to be sequentially open if for $x\in \mathfrak U$ such that $x_{n}\to x$ then there exists a natural number $\mathsf N$ such that $x_{n}\in \mathfrak U$ for all $n> \mathsf N.$ 
\end{definition}

\begin{proposition}
Let $(\mathsf X,d_{\Theta})$ be an extended cone $b-$metric space. Then the sequential topology $\sigma$ and the topology $\sigma_{d_{\Theta}}$ induced by $d_{\Theta}$ coincide.  
\end {proposition}

\begin{proof}
Suppose $\mathfrak U\in \sigma.$ Suppose $\mathfrak U\notin \sigma_{d_{\Theta}}.$ Then there is some $x\in \mathfrak U$ and $u_{1}\gg 0$ such that $\mathsf B (x,u_{1})$ is not contained in $\mathfrak U.$ Let $x_{n}\in \mathsf B (x,u_{1})$ such that $x_{n}\notin \mathfrak U$ for every natural number $n.$ Then $d_{\Theta}(x_{n},x)\ll u_{1}.$ This implies that $x_{n}\to x$ in $(\mathfrak X,d_{\Theta}).$ Since $\mathfrak U\in \sigma.$ Then there exists a natural number $\mathfrak N$ such that $x_{n}\in \mathfrak U$ for all $n> \mathfrak N,$ a contradiction.
Next suppose that $\mathfrak U\in \sigma_{d_{\Theta}}.$ For all $x\in \mathfrak U$ such that $x_{n}\to x$ in $(\mathfrak X,\sigma_{\Theta}),$ we have $\mathsf B (x,u_{1})\subset \mathfrak U$ for some $u_{1}\gg 0.$ This implies there exists a natural number $\mathfrak N_{0}$ such that $d_{\theta}(x_{n},x)\ll u_{1}$ for all $n\geq \mathfrak N_{0}.$ Hence $x_{n}\in \mathfrak U$ for all $n\geq \mathfrak N_{0}.$ Thus, $\mathfrak U\in \sigma.$
\end{proof}

\section{Kannan-type contractions in extended cone $b-$metric space}
\begin{proposition}
Let $(\mathsf{X},d_{\Theta})$ be an extended cone $b-$metric space and $\Theta:\mathsf X \times \mathsf X \times \mathsf X \to [1,\infty)$ be a map. If there exists $s\in [0,1)$ such that the sequence $\{x_{n_{1}}\}$ satisfies $\displaystyle\lim_{n_{1},n_{2}\to\infty} \Theta(x_{n_{1}},x_{n_{2}},x_{n_{1}+1})<1/s$ and 
\begin{equation}
0<d_{\Theta}(x_{n_{1}},x_{n_{1}+1})\leq sd_{\Theta}(x_{n_{1}-1},x_{n_{1}})
\end{equation}
for $n_{1}\in \mathbb{N},$ then the sequence $\{x_{n_{1}}\}$ is a Cauchy sequence.  
\end{proposition}

\begin{proof}
Let $\{x_{n_{1}}\}$ be a sequence in $\mathsf X.$ Now,
\begin{eqnarray}
d_{\Theta}(x_{n_{1}},x_{n_{1}+1})&\leq&sd_{\Theta}(x_{n_{1}-1},x_{n_{1}})\nonumber\\
                                 &\leq&s^{2}d_{\Theta}(x_{n_{1}-2},x_{n_{1}-1})\nonumber\\
																 &\vdots&\nonumber\\
																 &\leq&s^{n_{1}}d_{\Theta}(x_{0},x_{1}).
\end{eqnarray}
Since $s\in [0,1),$ we see that
\begin{equation}
\displaystyle\lim_{n_{1}\to\infty} d_{\Theta}(x_{n_{1}},x_{n_{1}+1})=0. 
\end{equation}
Using inequality  , for $n_{2}\geq n_{1},$ we get
\begin{eqnarray*}
d_{\Theta}(x_{n_{1}},x_{n_{2}})&\leq&\Theta(x_{n_1},x_{n_2},x_{n_{1}+1})(d_{\Theta}(x_{n_{1}},x_{n_{1}+1})+d_{\Theta}(x_{n_{1}+1},x_{n_{2}}))\\
                               &=&\Theta(x_{n_1},x_{n_2},x_{n_{1}+1}) d_{\Theta}(x_{n_{1}},x_{n_{1}+1})+\Theta(x_{n_1},x_{n_2},x_{n_{1}+1})d_{\Theta}(x_{n_{1}+1},x_{n_{2}})\\
															&\leq&\Theta(x_{n_1},x_{n_2},x_{n_{1}+1}) s^{n_{1}} d_{\Theta}(x_{0},x_{1})\\
															&+&\Theta(x_{n_{1}},x_{n_{2}},x_{n_{1}+1})\Theta(x_{n_{1}+1},x_{n_{2}},x_{n_{1}+2})(d_{\Theta}(x_{n_{1}+1},x_{n_{1}+2}),d_{\Theta}(x_                                 {n_{1}+2},x_{n_{2}}))\\
															&\leq&\Theta(x_{n_1},x_{n_2},x_{n_{1}+1}) s^{n_{1}} d_{\Theta}(x_{0},x_{1})\\
															&+&\Theta(x_{n_{1}},x_{n_{2}},x_{n_{1}+1})\Theta(x_{n_{1}+1},x_{n_{2}},x_{n_{1}+2}) s^{n_{1}+1} d_{\Theta}(x_{0},x_{1})\\
															&+&\ldots+\Theta(x_{n_{1}},x_{n_{2}},x_{n_{1}+1})\Theta(x_{n_{1}+1},x_{n_{2}},x_{n_{1}+2})\ldots\Theta(x_{n_{2}-2},x_{n_{2}},x_{n_{2}                              -1})s^{n_{2}-1}d_{\Theta}(x_{0},x_{1})\\
															&\leq&[\Theta(x_{n_1},x_{n_2},x_{n_{1}+1}) s^{n_{1}}\\
															&+&\Theta(x_{n_{1}},x_{n_{2}},x_{n_{1}+1})\Theta(x_{n_{1}+1},x_{n_{2}},x_{n_{1}+2})s^{n_{1}+1}\\
															&+&\ldots+\Theta(x_{n_{1}},x_{n_{2}},x_{n_{1}+1})\Theta(x_{n_{1}+1},x_{n_{2}},x_{n_{1}+2})\ldots\Theta(x_{n_{2}-2},x_{n_{2}},x_{n_{2}                                 -1})s^{n_{2}-1}]d_{\Theta}(x_{0},x_{1}).
\end{eqnarray*}
Since $\displaystyle\lim_{n_{1},n_{2}\to\infty} \Theta(x_{n_{1}},x_{n_{2}},x_{n_{1}+1})s<1.$ We see that by ratio test the series $\displaystyle\sum_{n_{1}=1}^{\infty}s^{n_{1}}\displaystyle\prod_{j=1}^{n_{1}}\Theta(x_{j},x_{n_{2}},x_{j+1})$ converges. Let $\mathcal S=\displaystyle\sum_{n_{1}=1}^{\infty}s^{n_{1}}\displaystyle\prod_{j=1}^{n_{1}}\Theta(x_{j},x_{n_{2}},x_{j+1})$ and $\mathcal S_{n_{1}}=\displaystyle\sum_{i=1}^{n_{1}}s^{i}\displaystyle\prod_{j=1}^{i}\Theta(x_{j},x_{n_{2}},x_{j+1}).$ Therefore,
\begin{equation}
d_{\Theta}(x_{n_{1}},x_{n_{2}})\leq d_{\Theta}(x_{1},x_{0})[\mathcal S_{n_{2}-1}-\mathcal S_{n_{1}-1}]
\end{equation}
Letting $n_{1}\to \infty,$ we obtain the desired result.
\end{proof}

\begin{theorem}\label{thm:9}
Let $(\mathsf X,d_{\Theta})$ be a complete extended cone $b-$metric space and $P$ be a cone in $E.$ Let $f:\mathsf X\to \mathsf X$ be a mapping that satisfies:
\begin{equation}
d_{\Theta}(f(x),f(y))\leq s[d_{\Theta}(x,f(x))+d_{\Theta}(y,f(y))] + td_{\Theta}(y,f(x)),\forall x,y\in \mathsf X,
\end{equation}
where $s\in (0,\frac{1}{2})$ and $t\in [0,1).$
Suppose that 
\begin{equation}
\displaystyle\sup_{n_{2}\geq 1}\lim_{n_{1}\to \infty}\Theta(x_{n_{1}},x_{n_{2}},x_{n_{1}+1})<\frac{1-s}{s},\forall x_{0}\in \mathsf X,
\end{equation}
such that $x_{n_{1}}=f^{n_{1}}(x_{0}),n_{1}\in \mathbb{N}.$ Then $f$ has a unique fixed point $u$ in $X.$ Moreover, for each $x\in \mathsf X,$ the sequence $\{f^{n_{1}}(x)\}$ is convergent to $u,$ and 
\begin{equation}
d_{\Theta}(f(x_{n_{1}}),u)\leq \frac{s}{1-t}\left(\frac{s}{1-s}\right)^{n_{1}}d_{\Theta}(f(x_{0}),x_{0}),~n_{1}=0,1,2,\ldots
\end{equation}
\end{theorem}

\begin{proof}
Let $x_{0}\in \mathsf X$ be arbitrary. Define
\begin{equation}
x_{n_{1}+1}=f(x_{n_{1}})=f^{n_{1}+1}(x_{0}).
\end{equation}
Clearly, $x_{n_{1}}$ is a fixed point of $f$ if $x_{n_{1}}=x_{n_{1}+1},$ for some $n_{1}\in \mathbb{N}.$ If not, suppose that $x_{n_{1}}$ and $x_{n_{1}+1}$ are distinct points in $\mathsf X$ for each $n_{1}\geq 0.$ Since
\begin{equation}
d_{\Theta}(x_{n_{1}},x_{n_{1}+1})=d_{\Theta}(f(x_{n_{1}-1}),f(x_{n_{1}}))
\end{equation}
we have
\begin{equation}
d_{\Theta}(f(x_{n_{1}-1}),f(x_{n_{1}}))\leq s[d_{\Theta}(x_{n_{1}-1},f(x_{n_{1}-1}))+d_{\Theta}((x_{n_{1}}),f(x_{n_{1}}))]+td_{\Theta}(x_{n_{1}},f(x_{n_{1}-1})).
\end{equation}
Then
\begin{equation}
d_{\Theta}(x_{n_{1}},x_{n_{1}+1})\leq s[d_{\Theta}(x_{n_{1}-1},x_{n_{1}})+d_{\Theta}(x_{n_{1}},x_{n_{1}+1})]+td_{\Theta}(x_{n_{1}},x_{n_{1}}).
\end{equation}
\end{proof}
Therefore,
\begin{equation}
d_{\Theta}(x_{n_{1}},x_{n_{1}+1})\leq \left( \frac{s}{1-s}\right) d_{\Theta}(x_{n_{1}-1},x_{n_{1}}).
\end{equation}
Proceeding in the similar manner, we see that 
\begin{equation}\label{eq:15}
d_{\Theta}(x_{n_{1}},x_{n_{1}+1})\leq {\left( \frac{s}{1-s}\right)}^{n}d_{\Theta}(x_{0},x_{1}),
\end{equation}
and
\begin{equation}
d_{\Theta}(f(x_{n_{1}-1}),f(x_{n_{1}})\leq {\left( \frac{s}{1-s}\right)}^{n}d_{\Theta}(x_{0},f(x_{0})).
\end{equation}
Suppose $n_{1},n_{2}$ are natural numbers such that $n_{2}>n_{1}.$ By applying triangular inequality, we have
\begin{eqnarray}
d_{\Theta}(x_{n_{1}},x_{n_{2}})&\leq&\Theta(x_{n_{1}},x_{n_{2}},x_{n_{1}+1})[d_{\Theta}(x_{n_{1}},x_{n_{1}+1})+d_{\Theta}(x_{n_{1}+1},x_{n_{2}})]\nonumber\\
                              &\leq&\Theta(x_{n_{1}},x_{n_{2}},x_{n_{1}+1})d_{\Theta}(x_{n_{1}},x_{n_{1}+1})\nonumber\\
															&+&\Theta(x_{n_{1}},x_{n_{2}},x_{n_{1}+1})\Theta(x_{n_{1}+1},x_{n_{2}},x_{n_{1}+2})[d_{\Theta}(x_{n_{1}+1},x_{n_{1}+2})+d_{\Theta}(x_                              {n_{1}+2},x_{n_{2}})]\nonumber\\
															&\leq&\Theta(x_{n_{1}},x_{n_{2}},x_{n_{1}+1})d_{\Theta}(x_{n_{1}},x_{n_{1}+1})\nonumber\\
															&+&\Theta(x_{n_{1}},x_{n_{2}},x_{n_{1}+1})\Theta(x_{n_{1}+1},x_{n_{2}},x_{n_{1}+2})d_{\Theta}(x_{n_{1}+1},x_{n_{1}+2})\nonumber\\
															&+&\ldots+\Theta(x_{n_{1}},x_{n_{2}},x_{n_{1}+1})\Theta(x_{n_{1}+1},x_{n_{2}},x_{n_{1}+2})\nonumber\\
															& &\Theta(x_{n_{1}+2},x_{n_{2}},x_{n_{1}+3})\ldots \Theta(x_{n_{2}-2},x_{n_{2}},x_{n_{2}-1})d_{\Theta}(x_{n_{2}-1},x_{n_{2}})
\end{eqnarray}
Since
\begin{eqnarray}
d_{\Theta}(x_{n_{1}},x_{n_{1}+1})&\leq&\left(\frac{s}{1-s}\right)^{n_{1}}d_{\Theta}(x_{0},x_{1}),n\geq 0,
\end{eqnarray}
we have
\begin{eqnarray}
d_{\Theta}(x_{n_{1}},x_{n_{2}})&\leq&\Theta(x_{n_{1}},x_{n_{2}},x_{n_{1}+1})\left(\frac{s}{1-s}\right)^{n_{1}}d_{\Theta}(x_{0},x_{1})\nonumber\\
                               &+&\Theta(x_{n_{1}},x_{n_{2}},x_{n_{1}+1})\Theta(x_{n_{1}+1},x_{n_{2}},x_{n_{1}+2})\left(\frac{s}{1-s}\right)^{n_{1}+1}d_{\Theta}(x_                               {0},x_{1})\nonumber\\
															 &+&\ldots+\Theta(x_{n_{1}},x_{n_{2}},x_{n_{1}+1})\Theta(x_{n_{1}+1},x_{n_{2}},x_{n_{1}+2})\Theta(x_{n_{1}+2},x_{n_{2}},x_{n_{1}+3})                               \nonumber\\
															& &\ldots \Theta(x_{n_{2}-2},x_{n_{2}},x_{n_{2}-1}) \left(\frac{s}{1-s}\right)^{n_{2}-1}d_{\Theta}(x_{0},x_{1})\nonumber\\
															&\leq&[\Theta(x_{n_{1}},x_{n_{2}},x_{n_{1}+1})\left(\frac{s}{1-s}\right)^{n_{1}}\nonumber\\
															&+&\Theta(x_{n_{1}},x_{n_{2}},x_{n_{1}+1})\Theta(x_{n_{1}+1},x_{n_{2}},x_{n_{1}+2})\left(\frac{s}{1-s}\right)^{n_{1}+1}\nonumber\\
															&+&\ldots+\Theta(x_{n_{1}},x_{n_{2}},x_{n_{1}+1})\Theta(x_{n_{1}+1},x_{n_{2}},x_{n_{1}+2})\Theta(x_{n_{1}+2},x_{n_{2}},x_{n_{1}+3})                               \nonumber\\
															& &\ldots \Theta(x_{n_{2}-2},x_{n_{2}},x_{n_{2}-1}) \left(\frac{s}{1-s}\right)^{n_{2}-1}]d_{\Theta}(x_{0},x_{1}).
\end{eqnarray}
Moreover, 
\begin{equation}
\displaystyle\sup_{n_{2}\geq 1}\lim_{n_{1},n_{2}\to \infty} \Theta(x_{n_{1}},x_{n_{2}},x_{n_{1}+1})\frac{s}{1-s}<1,
\end{equation}
We see that the series $\displaystyle\sum_{n_{1}=1}^{\infty}\left(\frac{s}{1-s}\right)^{n_{1}}\prod_{j=1}^{n_{1}}\Theta(x_{j},x_{n_{2}},x_{j+1})$ is convergent for every natural number $n_{2}$ by ratio test.\\
 Next suppose $\mathcal{S}=\displaystyle\sum_{n_{1}=1}^{\infty}\left(\frac{s}{1-s}\right)^{n_{1}}\prod_{j=1}^{n_{1}}\Theta(x_{j},x_{n_{2}},x_{j+1})$ and
\begin{equation}
 \mathcal{S}_{n_{1}}=\displaystyle\sum_{i=1}^{n_{1}}\left(\frac{s}{1-s}\right)^{i}\prod_{j=1}^{i}\Theta(x_{j},x_{n_{2}},x_{j+1}).
\end{equation} 
Hence for $n_{2}>n_{1},$ using above inequality we have
\begin{equation}
d_{\Theta}(x_{n_{1}},x_{n_{2}})\leq d_{\Theta}(x_{0},x_{1})(\mathcal{S}_{n_{2}-1}-\mathcal S_{n_{1}-1})
\end{equation}
Then
\begin{equation}
\displaystyle\lim_{n\to \infty} d_{\Theta}(x_{n_{1}},x_{n_{2}})=0.
\end{equation}
Therefore, $\{x_{n}\}$ is a Cauchy sequence. Since $\mathsf{X}$ is complete, there exists $u\in \mathsf X$ such that $x_{n_{1}}\to u$ as $n_{1}\to \infty.$\\
Claim: $u$ is a fixed point of $f.$\\
Since $d_{\Theta}(f(x_{n_{1}}),f(u))\leq s[d_{\Theta}(x_{n_{1}},f(x_{n_{1}}))+d_{\Theta}(u,f(u))]+td_{\Theta}(u,f(x_{n_{1}})).$ In context of the previous supposition that $d_{\Theta}$ is continuous, taking limit $n_{1}\to \infty,$ we have
\begin{equation}
d_{\Theta}(u,f(u))\leq s d_{\Theta}(u,f(u)).
\end{equation}
This is only possible when $d_{\Theta}(u,f(u))=0.$ Hence $f(u)=u.$\\
Next we show that fixed point of $f$ is unique. For this, let $v$ be a fixed point of $f$ distinct from $u.$ Then
\begin{equation}
d_{\Theta}(u,v)=d_{\Theta}(f(u),f(v))\leq s[d_{\Theta}(u,f(u))+d_{\Theta}(v,f(v))]+td_{\Theta}(u,f(v)).
\end{equation}
We have
\begin{equation}
d_{\Theta}(u,v)\leq t d_{\Theta}(u,v).
\end{equation}
 This is only possible when $d_{\Theta}(u,v)=0.$ Hence $u$ is the unique fixed point of $f$ in $\mathsf X.$  
Also, we have
\begin{equation}
d_{\Theta}(f(x_{n_{1}-1}),f(x_{n_{1}}))\leq s[d_{\Theta}(f(x_{n_{1}-2}),f(x_{n_{1}-1}))+d_{\Theta}(f(x_{n_{1}-1}),f(x_{n_{1}}))]+td_{\Theta}(x_{n_{1}},x_{n_{1}}).
\end{equation}
This implies that
\begin{equation}
d_{\Theta}(f(x_{n_{1}-1}),f(x_{n_{1}}))\leq\left(\frac{s}{1-s}\right)d_{\Theta}(f(x_{n_{1}-2}),f(x_{n_{1}-1}))
\end{equation}
Further,
\begin{eqnarray}
d_{\Theta}(f(x_{n_{1}}),u)&\leq& s[d_{\Theta}(f(x_{n_{1}-1}),f(x_{n_{1}}))+d_{\Theta}(u,f(u))]+td_{\Theta}(u,f(x_{n_{1}}))\nonumber\\
                          &\leq& s d_{\Theta}(f(x_{n_{1}-1}),f(x_{n_{1}}))+td_{\Theta}(u,f(x_{n_{1}}))
\end{eqnarray}
 From (\ref{eq:15}) we have
\begin{equation}
d_{\Theta}(f(x_{n_{1}}),u) \leq \frac{s}{1-t}\left(\frac{s}{1-s}\right)^{n_{1}}d_{\Theta}(f(x_{0}),x_{0}),n\geq 0.
\end{equation}
Hence the proof.
\begin{theorem}\label{thm:14}
Let $(\mathsf{X},d_{\Theta})$ be a complete extended cone $b-$metric space, $d_{\Theta}$ be a continuous functional and $\mathcal{N}\neq \phi$ be a closed set contained in $\mathsf{X}.$ Suppose $f:\mathcal N\to \mathcal N$ be a mapping that satisfies 
\begin{equation}\label{eq:58}
d_{\Theta}(f(x),f(y))\leq s[d_{\Theta}(x,f(x))+d_{\Theta}(y,f(y))]+td_{\Theta}(y,f(x)),\forall x,y\in \mathcal{N},0\leq s,t 1
\end{equation}
and there exist real numbers $\gamma,\delta$ where $\gamma\in (0,1)$ and $\delta >0$ such that for arbitrary $x\in \mathcal N,$ there exists $x^{*}$ in $\mathcal N$ satisfying
\begin{eqnarray}
d_{\Theta}(x^{*},f(x^{*}))&\leq&\gamma d_{\Theta}(x,f(x)),\nonumber\\
d_{\Theta}(x^{*},x)&\leq&\delta d_{\Theta}(x,f(x)).
\end{eqnarray}
Moreover, for an arbitrary $x_{0}\in \mathcal N,$ suppose that $\{x_{n_{1}}=f^{n_{1}}(x_{0})\}$ satisfies 
\begin{equation}
\displaystyle\sup_{n_{2}\geq 1}\lim_{n_{1}\to \infty}\Theta(x_{n_{1}},x_{n_{1}+1},x_{n_{2}})<\frac{1}{\gamma}.
\end{equation}
Then $f$ has a unique fixed point.
\end{theorem}
\begin{proof}
Consider an arbitrary element $x_{0}\in \mathcal N.$ Let $\{x_{n_{1}}=f^{n_{1}}(x_{0})\}$ be a sequence in $\mathcal N.$ We see that
\begin{equation}
d_{\Theta}(f(x_{n_{1}+1}),x_{n_{1}+1})\leq \gamma (d_{\Theta}(f(x_{n_{1}}),x_{n_{1}})),
d_{\Theta}(f(x_{n_{1}+1}),x_{n_{1}+1})\leq \delta d_{\Theta}(f(x_{n_{1}}),x_{n_{1}}),n_{1}\geq 0.
\end{equation}
Moreover, 
\begin{eqnarray*}
d_{\Theta}(x_{n_{1}+1},x_{n_{1}})=d_{\Theta}(f(x_{n_{1}}),x_{n_{1}})\leq \delta d_{\Theta}(f(x_{n_{1}}),x_{n_{1}}),n_{1}\geq 0
\end{eqnarray*}
\begin{eqnarray}
\delta d_{\Theta}(f(x_{n_{1}}),x_{n_{1}})&\leq&\delta \gamma d_{\Theta}(f(x_{n_{1}-1}),x_{n_{1}-1})\nonumber\\
                                         &\leq&\delta \gamma^{2} d_{\Theta}(f(x_{n_{1}-2}),x_{n_{1}-2})\nonumber\\
																				 & \vdots&\nonumber\\
																				 &\leq &\delta \gamma^{n} d_{\Theta}(f(x_{0}),x_{0}).
\end{eqnarray}
Hence
\begin{equation} \label{eq:63}
d_{\Theta}(x_{n_{1}+1},x_{n_{1}})\leq \delta \gamma^{n}d_{\Theta}(f(x_{0}),x_{0}).
\end{equation}
Let $n_{1},n_{2}$ be two fixed natural numbers such that $n_{2}>n_{1}.$ Using triangular inequality we have
\begin{eqnarray}
d_{\Theta}(x_{n_{1}},x_{n_{2}})&\leq&\Theta(x_{n_{1}},x_{n_{2}},x_{n_{1}+1})[d_{\Theta}(x_{n_{1}},x_{n_{1}+1})+d_{\Theta}(x_{n_{1}+1},x_{n_{2}})]\nonumber\\
                              &\leq&\Theta(x_{n_{1}},x_{n_{2}},x_{n_{1}+1})d_{\Theta}(x_{n_{1}},x_{n_{1}+1})\nonumber\\
															&+&\Theta(x_{n_{1}},x_{n_{2}},x_{n_{1}+1})\Theta(x_{n_{1}+1},x_{n_{2}},x_{n_{1}+2})[d_{\Theta}(x_{n_{1}+1},x_{n_{1}+2})+d_{\Theta}(x_                              {n_{1}+2},x_{n_{2}})]\nonumber\\
															&\leq&\Theta(x_{n_{1}},x_{n_{2}},x_{n_{1}+1})d_{\Theta}(x_{n_{1}},x_{n_{1}+1})\nonumber\\
															&+&\Theta(x_{n_{1}},x_{n_{2}},x_{n_{1}+1})\Theta(x_{n_{1}+1},x_{n_{2}},x_{n_{1}+2})d_{\Theta}(x_{n_{1}+1},x_{n_{1}+2})\nonumber\\
															&+&\ldots+\Theta(x_{n_{1}},x_{n_{2}},x_{n_{1}+1})\Theta(x_{n_{1}},x_{n_{2}},x_{n_{1}+2})\nonumber\\
															& &\Theta(x_{n_{1}+2},x_{n_{2}},x_{n_{1}+3})\ldots \Theta(x_{n_{2}-2},x_{n_{2}},x_{n_{2}-1})d_{\Theta}(x_{n_{2}-1},x_{n_{2}})
\end{eqnarray}
From (\ref{eq:63}), we have
\begin{eqnarray} 
d_{\Theta}(x_{n_{1}},x_{n_{2}})&\leq&[\Theta(x_{n_{1}},x_{n_{2}},x_{n_{1}+1}) \gamma^{n_{1}}\nonumber\\
															&+&\Theta(x_{n_{1}},x_{n_{2}},x_{n_{1}+1})\Theta(x_{n_{1}+1},x_{n_{2}},x_{n_{1}+2}) \gamma^{n_{1}+1} d_{\Theta}(f(x_{0},x_{0})) \nonumber\\
															&+&\ldots+\Theta(x_{n_{1}},x_{n_{2}},x_{n_{1}+1})\Theta(x_{n_{1}+1},x_{n_{2}},x_{n_{1}+2})\nonumber\\
															& &\Theta(x_{n_{1}+2},x_{n_{2}},x_{n_{1}+3})\ldots \Theta(x_{n_{2}-2},x_{n_{2}},x_{n_{2}-1})\gamma^{n_{2}-1}]\nonumber\\
															& &\delta d_{\Theta}(f(x_{0}),x_{0})
\end{eqnarray}
Since $\displaystyle\sup_{n_{2}\geq 1}\lim_{n_{1},n_{2}\to \infty}\Theta(x_{n_{1}},x_{n_{1}+1},x_{n_{2}})\gamma<1,$ the series $\displaystyle\sum_{n_{1}=1}^{\infty}\gamma^{n_{1}}\prod_{j=1}^{n_{1}}\Theta(x_{j},x_{n_{2}},x_{j+1})$ is convergent for every natural number $n_{2}$ by ratio test.\\ 
Suppose $\mathcal{S}=\displaystyle\sum_{n_{1}=1}^{\infty}\gamma^{n_{1}}\prod_{j=1}^{n_{1}}\Theta(x_{j},x_{n_{2}},x_{j+1})$ and
\begin{equation}
 \mathcal{S}_{n_{1}}=\displaystyle\sum_{i=1}^{n_{1}}\gamma^{i}\prod_{j=1}^{i}\Theta(x_{j},x_{n_{2}},x_{j+1}).
\end{equation} 
Hence for $n_{2}>n_{1},$ using above inequality we have
\begin{equation}
d_{\Theta}(x_{n_{1}},x_{n_{2}})\leq d_{\Theta}(x_{0},x_{1})(\mathcal{S}_{n_{2}-1}-\mathcal S_{n_{1}-1})\delta.
\end{equation}
Suppose $n_{1}\to \infty.$ Therefore, $\{x_{n}\}$ is a Cauchy sequence. Since $\mathcal N$ is complete, there exists $u\in \mathcal N$ such that $x_{n_{1}}\to u$ as $n_{1}\to \infty.$\\ 
We shall show that $u$ is a fixed point of $f.$ By (\ref{eq:58}), we see that 
\begin{eqnarray*}
d_{\Theta}(f(x_{n_{1}}),f(u))&\leq& s[d_{\Theta}(x_{n_{1}},f(x_{n_{1}}))+d_{\Theta}(u,f(u))]+td_{\Theta}(u,f(x_{n_{1}})).
\end{eqnarray*} 
Therefore, 
\begin{eqnarray*}
d_{\Theta}(x_{n_{1}+1},f(u))&\leq& s[d_{\Theta}(x_{n_{1}},x_{n_{1}+1})+d_{\Theta}(u,f(u))]+td_{\Theta}(u,x_{n_{1}+1}).
\end{eqnarray*}
Letting $n\to \infty$ and using the continuity of $d_{\Theta}$ we get
\begin{equation}
d_{\Theta}(u,f(u))\leq sd_{\Theta}(u,f(u)).
\end{equation}
This is only possible when $f(u)=u.$ Next we shall show that $f$ has a unique fixed point. If possible, suppose that $v$ is a  fixed point of $f$ distinct from $u.$ Then
\begin{eqnarray}
0&<&d_{\Theta}(u,v)\nonumber\\
 &=&d_{\Theta}(f(u),f(v))\nonumber\\
 &\leq&s[d_{\Theta}(u,f(u))+d_{\Theta}(v,f(v))]+td_{\Theta}(v,f(u))\nonumber\\
 &=&td_{\Theta}(v,u).
\end{eqnarray}
Now, $d_{\Theta}(u,v)\leq t d_{\Theta}(u,v)$ is not possible. Therefore, $f$ has a unique fixed point $u$ in $\mathsf X.$
\end{proof}

\begin{remark}
Theorem \ref{thm:9} can be proved in extended cone $b-$metric space using the following:
\begin{eqnarray*}
d_{\Theta}(u,f(u))&\leq&\beta d_{\Theta}(x,f(x))\\
d_{\Theta}(u,x)&\leq&\gamma d_{\Theta}(y,f(y)).
\end{eqnarray*}
\end{remark}

\begin{proof}
Let $x\in \mathsf X$ be arbitrary and $u=f(x).$ Then we have 
\begin{eqnarray*}
d_{\Theta}(u,f(u))&=&d_{\Theta}(f(x),f(u))\\
                  &\leq&s[d_{\Theta}(x,f(x))+d_{\Theta}(u,f(u))]+td_{\Theta}(u,f(x))\\
\implies d_{\Theta}(u,f(u))&\leq&\left(\frac{s}{1-s}\right)d_{\Theta}(x,f(x)),
\end{eqnarray*}
where $\left(\frac{s}{1-s}\right)<1$ and $d_{\Theta}(u,x)=d_{\Theta}(f(x),x).$ Let $x_{0}\in \mathsf X$ be arbitrary. Next define a sequence $\{x_{n_{1}+1}=f(x_{n_{1}})\}.$ Using Theorem \ref{thm:14}, the above sequence is convergent. Hence $x_{n_{1}}\to u$ as $n_{1}\to \infty.$ Therefore, $f(u)=u.$ Moreover for every $x\in \mathsf X,$
\begin{eqnarray*}
d_{\Theta} (f(x_{n_{1}-1}),f(x_{n_{1}}))&\leq&s[d_{\Theta} (f(x_{n_{1}-2}),f(x_{n_{1}-1}))+d_{\Theta} (f(x_{n_{1}-1}),f(x_{n_{1}}))]\\
                                        &+&td_{\Theta}(x_{n_{1}},f(x_{n_{1}-1}))\\
																			&\leq&\left(\frac{s}{1-s}\right)d_{\Theta}(f(x_{n_{1}-2}),f(x_{n_{1}-1}))
\end{eqnarray*}
\begin{eqnarray*}
d_{\Theta} (f(x_{n_{1}}),u)		&\leq&s[d_{\Theta} (f(x_{n_{1}-1}),f(x_{n_{1}}))+d_{\Theta} (u,f(u))]\\
                             &+&td_{\Theta}(u,f(x_{n_{1}}))\\
														&\leq&sd_{\Theta} (f(x_{n_{1}-1}),f(x_{n_{1}}))+td_{\Theta}(u,f(x_{n_{1}}))\\
															&\leq&\frac{s}{1-t}	{\left(\frac{s}{1-s}\right)}^{n}d_{\Theta}(f(x),x)),n_{1}\geq 0.				
\end{eqnarray*} 
\end{proof}

\begin{theorem}
Let $(\mathsf{X},d_{\Theta})$ be a complete extended cone $b-$metric space such that $d_{\Theta}$ is a continuous functional. Suppose that the map $f:\mathsf X\to \mathsf X$ satisfies
\begin{eqnarray}\label{eq:74}
d_{\Theta}(f(x),f(y))&\leq& \beta d_{\Theta}(x,f(x))+\gamma d_{\Theta}(y,f(y))+\delta d_{\Theta}(x,y) \nonumber\\
                     &+& td_{\Theta}(x,f(y)),\forall x,y\in \mathsf X,
\end{eqnarray}  
where $\beta,\gamma,\delta,t\in \mathbb{R}_{\geq 0}$ such that $\beta+\gamma+\delta+t<1$ and $\gamma+\delta>0.$ Suppose that for any $x_{0}\in \mathcal N,$ 
\begin{equation}
\displaystyle\sup_{n_{2}\geq 1}\lim_{n_{1}\to \infty}\Theta(x_{n_{1}},x_{n_{2}},x_{n_{1}+1})<\frac{1}{q},
\end{equation}
where $q=\left(\frac{\gamma+\delta}{1-\beta}\right)$ and $x_{n_{1}}=f^{n_{1}}(x_{0}).$ Then $f$ has a unique fixed point.
\end{theorem}
\begin{proof}
Let $x_{0}\in \mathsf X$ be arbitrary. Consider the sequence $\{f^{n_{1}}(x_{0})\}.$ Put $x=f^{n_{1}-1}(x_{0})=f(x_{n_{1}-1})=x_{n_{1}}$ and $y=f^{n_{1}-2}(x_{0})=
f(x_{n_{1}-2})=x_{n_{1}-1}$ in (\ref{eq:74}), we get 
\begin{eqnarray}
d_{\Theta}(f(x_{n_{1}}),f(x_{n_{1}-1}))&\leq& \beta d_{\Theta}(f(x_{n_{1}-1}),f(x_{n_{1}}))+\gamma d_{\Theta}(f(x_{n_{1}-2}),f(x_{n_{1}-1}))\nonumber\\
                                       &+&\delta d_{\Theta}(f(x_{n_{1}-1}),f(x_{n_{1}-2})) \nonumber\\
                                       &+& td_{\Theta}(f(x_{n_{1}-1}),f(x_{n_{1}-1})).
\end{eqnarray} 
That is,
\begin{equation}
(1-\beta)d_{\Theta}(f(x_{n_{1}}),f(x_{n_{1}-1}))\leq(\gamma+\delta)d_{\Theta}(f(x_{n_{1}-1}),f(x_{n_{1}-2})).
\end{equation}
Therefore,
\begin{equation}
d_{\Theta}(f(x_{n_{1}}),f(x_{n_{1}-1}))\leq\left(\frac{\gamma+\delta}{1-\beta}\right)d_{\Theta}(f(x_{n_{1}-1}),f(x_{n_{1}-2})).
\end{equation}
Also,
\begin{eqnarray*}
d_{\Theta}(f(x_{n_{1}}),f(x_{n_{1}-1}))&\leq&q d_{\Theta}(f(x_{n_{1}-1}),f(x_{n_{1}-2}))\\
                                       &\leq&q^{2} d_{\Theta}(f(x_{n_{1}-2}),f(x_{n_{1}-3}))\\
                                       &\vdots&\\
																			 &\leq&q^{n-1} d_{\Theta}(f(x_{1}),f(x_{n_{0}})),\forall n_{1}>1.
\end{eqnarray*}
Hence we have
\begin{equation*}
d_{\Theta}(f(x_{n_{1}}),f(x_{n_{1}-1}))\leq q^{n}d_{\Theta}(x_{0},x_{1}),\forall n_{1}\in \mathbb{N}. 
\end{equation*}
By given hypothesis we see that $q=\left(\frac{\gamma+\delta}{1-\beta}\right)<1.$ Proceeding similarly as in Theorem \ref{thm:14} we see that $\{x_{n_{1}}\}$ is a Cauchy sequence. Since $\mathsf X$ is complete. Then there exists $u\in \mathsf X$ such that $f^{n_{1}}(x_{0})\to u$ as $n_{1}\to \infty.$ To show that $u$ is fixed point of $f,$ substitute $x=f^{n_{1}}(x_{0})$ and $y=u$  in (\ref{eq:74}). We get 
\begin{eqnarray}
d_{\Theta}(f^{n_{1}+1}(x_{0}),f(u))&\leq& \beta d_{\Theta}(f^{n_{1}}(x_{0}),f^{n_{1}+1}(x_{0}))+\gamma d_{\Theta}(u,f(u))+\delta d_{\Theta}(f^{n_{1}}(x_{0}),u)\nonumber\\
                     &+& td_{\Theta}(f^{n_{1}}(x_{0}),f(u)).
\end{eqnarray}  
Hence
\begin{eqnarray}
d_{\Theta}(x_{n_{1}+2},f(u))&\leq&\beta d_{\Theta}(f^{n_{1}}(x_{0}),f^{n_{1}+1}(x_{0}))+\gamma d_{\Theta}(u,f(u))+\delta d_{\Theta}(x_{n_{1}},u)\nonumber\\
                         &+& td_{\Theta}(f(u),x_{n_{1}+1}).
\end{eqnarray}
That is,
\begin{eqnarray}
\displaystyle\lim_{n\to\infty}d_{\Theta}(x_{n_{1}+2},f(u))&\leq&\displaystyle\lim_{n\to \infty} \beta d_{\Theta}(f^{n_{1}}(x_{0}),f^{n_{1}+1}(x_{0}))+\gamma d_{\Theta}(u,f(u))+\delta d_{\Theta}(x_{n_{1}},u)\nonumber\\
&+& td_{\Theta}(f(u),x_{n_{1}+1}).
\end{eqnarray}
We get
\begin{eqnarray*}
d_{\Theta}(u,f(u))&\leq&(\gamma+t)d_{\Theta}(u,f(u)),
\end{eqnarray*}
which is only possible when $u=f(u).$ \\
To prove that $f$ has a unique fixed point let $v$ be a fixed point of $f$ distinct from $u.$ Then by \ref{eq:74}, we have
 \begin{eqnarray}\label{eq:75}
d_{\Theta}(f(v),f(u))&\leq& \beta d_{\Theta}(v,f(v))+\gamma d_{\Theta}(u,f(u))+\delta d_{\Theta}(v,u) \nonumber\\
                     &+& td_{\Theta}(f(u),v)\nonumber\\
\implies d_{\Theta}(v,u)&\leq&(t+\delta) d_{\Theta}(v,u) 
\end{eqnarray}
which is not possible. Therefore, $f$ has a unique fixed point.
\end{proof}

\begin{definition}
Let $(\mathsf{X},d_{\Theta})$ be an extended cone $b-$metric space. A mapping $f:\mathsf X\to \mathsf X$ is said to be asymptotically regular if $d_{\Theta}(f^{n_{1}+1}(x),f^{n_{1}}(x))\to 0$ as $n\to \infty,$ for each $x\in \mathsf X.$ 
\end{definition}

\begin{theorem}
Let $(\mathsf{X},d_{\Theta})$ be a complete extended cone $b-$metric space such that $d_{\Theta}$ is a continuous functional. Let $f:\mathsf X\to \mathsf X$ be an asymptotically regular self mapping 
\begin{eqnarray}
d_{\Theta}(f(x),f(y))&\leq&s[d_{\Theta}(x,f(x))+d_{\Theta}(y,f(y))],\forall x,y\in \mathsf X.
\end{eqnarray}  
Then $f$ has a unique fixed point in $u\in \mathsf X.$ 
\end{theorem}
\begin{proof}
Let $x\in \mathsf X$ and define $x_{n_{1}}=f^{n_{1}}(x).$ Let $n_{1}$ and $n_{2}$ be two fixed natural numbers such that $n_{2}>n_{1},$ then by the definition of asymptotic regularity, we have
\begin{eqnarray}
d_{\Theta}(f^{n_{1}+1}(x),f^{n_{2}+1}(x))&\leq&s[d_{\Theta}(f^{n_{1}}(x),f^{n_{1}+1}(x))\nonumber\\
&+&d_{\Theta}(f^{n_{2}}(x),f^{n_{2}+1}(x))]\to 0~\mbox{as}~n\to\infty.
\end{eqnarray}
Therefore,$\{f^{n_{1}}(x)\}$ is a Cauchy sequence. Since $X$ is  complete. Then there is $u\in \mathsf X$ such that 
\begin{eqnarray}
\displaystyle\lim_{n_{1}\to \infty} f^{n_{1}}(x)=u.
\end{eqnarray}
Next we shall show that $u$ is a fixed point of $f$ in the following manner:
\begin{eqnarray}
d_{\Theta}(f(x_{n_{1}}),f(u))&\leq&s(d_{\Theta}(x_{n_{1}},f(x_{n_{1}})))+d_{\Theta}(x_{n_{2}},f(x_{n_{2}})).
\end{eqnarray}
That is,
\begin{eqnarray}
d_{\Theta}(f(x_{n_{1}}),f(u))&\leq&s(d_{\Theta}(x_{n_{1}},x_{n_{1}+1}))+d_{\Theta}(u,f(u))
\end{eqnarray}
 Let $n\to \infty$ and by using asymptotically regular of $f$, we have
\begin{eqnarray}
d_{\Theta}(u,f(u))&\leq&s d_{\Theta}(u,f(u))
\end{eqnarray}
which holds when $f(u)=u.$ 
To prove that $u$ is unique fixed point of $f,$ let $v$ be a fixed point of $f$ distinct from $u.$ We get 
\begin{eqnarray}
d_{\Theta}(u,v)&=&d_{\Theta}(f(u),f(v))\\
               &\leq&s(d_{\Theta}(u,f(v))+d_{\Theta}(v,f(v)))
\end{eqnarray} 
which holds when $d_{\Theta}(u,v)=0.$ This implies that $u=v.$ Therefore, $u$ is the unique fixed point of $f.$ Moreover, for each $x\in \mathsf X,\{f^{n_{1}}(x)\}$ is convergent to $u.$
\end{proof}

\begin{remark}
The condition on $\Theta(x_{n_{1}}, x_{n_{1}+1},x_{n_{2}})$ can be dropped if the map is asymptotically regular. 
\end{remark}

\begin{theorem}
Let $(\mathsf{X},d_{\Theta})$ be a complete extended cone $b-$metric space such that $d_{\Theta}$ is a continuous functional. Consider an asymptotically regular mapping $f:\mathsf X\to \mathsf X$ such that $d_{\Theta}$ is a continuous functional such that there exists $t\in (0,1)$ such that 
\begin{eqnarray}\label{eq:93}
d_{\Theta}(f(x),f(y))&\leq&t[d_{\Theta}(x,f(x))+d_{\Theta}(y,f(y))+d_{\Theta}(x,y)],\forall x,y\in \mathsf X. 
\end{eqnarray}
Then $f$ has a unique fixed point $u\in \mathsf X$ unless
\begin{eqnarray}\label{eq:94}
\displaystyle\lim_{n\to \infty} \frac{t+t\Theta(x_{n_{1}},x_{n_{2}},x_{n_{1}+1})\Theta(x_{n_{1}+1},x_{n_{2}},x_{n_{2}+1})}{1-t\Theta(x_{n_{1}},x_{n_{2}},x_{n_{1}+1})\Theta(x_{n_{1}+1},x_{n_{2}},x_{n_{2}+1})}
\end{eqnarray}
exists for $x_{n}=f^{n_{1}}(x),n_{2}>n_{1}$ and $x\in \mathsf{X}$ is arbitrary.
\end{theorem}
\begin{proof}
Let $x\in \mathsf X$ and $x_{n_{1}}=f^{n_{1}}(x).$ Let $n_{1},n_{2}$ be fixed natural numbers such that $n_{2}>n_{1}.$ Then by (\ref{eq:93}), we have
\begin{eqnarray*}
d_{\Theta}(f^{n_{1}+1}(x),f^{n_{2}+1}(x)&\leq& t[d_{\Theta}(f^{n_{1}}(x),f^{n_{1}+1}(x))+d_{\Theta}(f^{n_{2}}(x),f^{n_{2}+1}(x))+d_{\Theta}(f^{n_{1}}(x),f^{n_{2}}(x))]\\
&\leq&t[d_{\Theta}(f^{n_{1}}(x),f^{n_{1}+1}(x))+d_{\Theta}(f^{n_{2}}(x),f^{n_{2}+1}(x))]\\
&+&t\Theta(x_{n_{1}},x_{n_{2}},x_{n_{1}+1})[d_{\Theta}(f^{n_{1}}(x),f^{n_{1}+1}(x))+d_{\Theta}(f^{n_{1}+1}(x),f^{n_{2}}(x))]\\
&\leq&t[d_{\Theta}(f^{n_{1}}(x),f^{n_{1}+1}(x))+d_{\Theta}(f^{n_{2}}(x),f^{n_{2}+1}(x))]\\
&+&t\Theta(x_{n_{1}},x_{n_{2}},x_{n_{1}+1})d_{\Theta}(f^{n_{1}}(x),f^{n_{1}+1}(x))\\
&+& t\Theta(x_{n_{1}},x_{n_{2}},x_{n_{1}+1})\Theta(x_{n_{1}+1},x_{n_{2}},x_{n_{2}+1})[d_{\Theta}(f^{n_{1}+1}(x),f^{n_{2}+1}(x))\\
&+&d_{\Theta}(f^{n_{2}+1}(x),f^{n_{2}}(x))]d_{\Theta} (f^{n_{1}+1}(x),f^{n_{2}+1}(x))\\
&\leq&\left(\frac{t+t\Theta(x_{n_{1}},x_{n_{2}},x_{n_{1}+1})}{1-t\Theta(x_{n_{1}},x_{n_{2}},x_{n_{1}+1})\Theta(x_{n_{1}+1},x_{n_{2}},x_{n_{2}+1})}\right)d_{\Theta}(f^{n_{1}}(x),f^{n_{1}+1}(x))\\
&+&\left(\frac{t+t\Theta(x_{n_{1}},x_{n_{2}},x_{n_{1}+1})\Theta(x_{n_{1}+1},x_{n_{2}},x_{n_{2}+1})}{1-t\Theta(x_{n_{1}},x_{n_{2}},x_{n_{1}+1})\Theta(x_{n_{1}+1},x_{n_{2}},x_{n_{2}+1})}\right)\\
& & d_{\Theta}(f^{n_{1}+1}(x),f^{n_{2}}(x))\\
& &\to 0~\mbox{as n}\to\infty.
\end{eqnarray*}
Therefore, $\{f^{n_{1}}(x)\}$ is a  Cauchy sequence. Since $X$ is complete, there exists $u\in \mathsf X$ such that 
\begin{equation}
\{f^{n_{1}}(x)\}\to u~as~n\to \infty.
\end{equation}
 Next using triangular inequality and (\ref{eq:93}), we have 
\begin{equation}
d_{\Theta}(f_{x_{n_{1}}},f(u))\leq t[d_{\Theta}(x_{n_{1}},f(x_{n_{1}}))+d_{\Theta}(u,f(u))+d_{\Theta}(x_{n_{1}},u)]
\end{equation}
Therefore,
\begin{equation}
d_{\Theta}(x_{n_{1}+1},f(u))\leq t[d_{\Theta}(x_{n_{1}},x_{n_{1}+1})+d_{\Theta}(u,f(u))+d_{\Theta}(x_{n_{1}},u)]
\end{equation}
Taking limit $n_{1}\to\infty,$ we have
\begin{equation}
\displaystyle\lim_{n_{1}\to \infty} (d_{\Theta}(x_{n_{1}+1},f(u))\leq\displaystyle\lim_{n\to \infty} t[d_{\Theta}(x_{n_{1}},x_{n_{1}+1})+d_{\Theta}(u,f(u))+d_{\Theta}(x_{n_{1}},u)]). 
\end{equation}
Therefore, $d_{\Theta} (u,f(u))\leq t d_{\Theta} (u,f(u)).$ This implies that $u=f(u).$ To prove that $f$ has a unique fixed point. Let $v$ be a fixed point of $f$ distinct from $u.$ Then
\begin{equation}
d_{\Theta}(f(u),f(v))\leq t[d_{\Theta}(v,f(v))+d_{\Theta} (u,f(u))] + d_{\Theta}(u,v)],t<1
\end{equation}
which is a contradiction. Therefore, $f$ has a unique fixed point. Hence $\{f^{n_{1}}(x)\}$ is convergent for each $x\in \mathsf{X}.$ 
\end{proof}
\bibliographystyle{amsplain}

\end{document}